\documentclass[a4paper]{amsart}
\usepackage{amsmath, amsfonts, amsthm, amssymb, xypic, hyperref, tikz,
bbm}
\usetikzlibrary{calc}

\newcommand{\Ext}{\operatorname{Ext}}
\newcommand{\End}{\operatorname{End}}
\newcommand{\soc}{\operatorname{soc}}
\newcommand{\rad}{\operatorname{rad}}
\newcommand{\HH}{\operatorname{HH}}
\newcommand{\Hom}{\operatorname{Hom}}

\newcommand{\im}{\operatorname{im}}

\newcommand{\ZZ}{\mathbb{Z}}
\newcommand{\CC}{\mathbb{C}}
\newcommand{\ssl}{\mathfrak{sl}}

\newcommand{\ad}{\operatorname{ad}}

\newcommand{\vl}{\varpi_1}
\newcommand{\vt}{\varpi_2}

\theoremstyle{plain}% default
\newtheorem{thm}{Theorem}[section]
\newtheorem{lem}[thm]{Lemma}
\newtheorem{prop}[thm]{Proposition}
\newtheorem{cor}[thm]{Corollary}
\theoremstyle{definition}
\newtheorem{defn}[thm]{Definition}

\author{Matthew Towers}
\title{Singular blocks of restricted $\ssl_3$}
\email{mjt43@le.ac.uk}
\date{\today}

\begin{document}

\begin{abstract}
  We compute generators and relations for the basic algebra of a
  non-semisimple singular block of the restricted enveloping algebra of
  $\ssl_3$ over an algebraically closed field of characteristic $p>3$. Working directly with
  the basic algebra we compute its centre and the internal degree zero
  part of its first Hochschild cohomology, and show its Verma modules are Koszul.
\end{abstract}

\maketitle
Let $A$ be a finite-dimensional algebra over an algebraically closed
field $k$.  The basic algebra of $A$ is a $k$-algebra all of whose
simple modules are one-dimensional which is Morita equivalent to $A$; it
can be thought of as the smallest algebra with the same representation
theory as $A$.
The small size of the basic algebra tends to make calculations
considerably easier.

This paper deals with restricted enveloping algebras of classical Lie
algebras.   While the basic algebras for non-semisimple blocks of
restricted $\ssl_2$ have long been known (and are reviewed in Section
\ref{hoch_section}), to the best of my knowledge no other nontrivial
examples have been computed, although some endomorphism algebras of projectives
are constructed in \cite[Chapter 19]{AJS}. Here we find generators and relations for the
basic algebras of non-semisimple singular blocks
of the restricted enveloping algebra $u$ of $\ssl_3(k)$ over an
algebraically closed field of
characteristic $p>3$.
The simple modules of
$u$ are labelled by their highest weights; singular here means that the
block has a simple module whose highest weight $\lambda$ satisfies
$\langle \lambda + \rho, \alpha\rangle=0 \mod p$ for some root $\alpha$,
where $\rho$ is the half-sum of the positive roots for $\ssl_3(k)$.

The method is as follows.  Thanks to work of Simon Riche \cite{Riche} we
know that each block of $u$ admits a $\mathbb{Z}_{\geq 0}$-grading
with respect to which it is a Koszul algebra.  We show in Section
\ref{verma_section} that the Verma
modules in our blocks are graded modules with respect to this grading.
This allows us to
use a version of the Brauer-type reciprocity of \cite{Holmes} to
find the graded structure of the projectives in our block, and hence the
Hilbert series for the basic algebra and of its Koszul dual.

In Section \ref{sec:action} we determine the action of $G=SL_3(k)$ on
the Ext-groups between simples in our blocks.  We then exploit this
to calculate generators and relations for the basic algebra in Theorem
\ref{presentation}.  In Section
\ref{hoch_section} we use
this presentation of the basic algebra to calculate its centre and the
internal degree zero part
of its first
Hochschild cohomology.  Finally in Section \ref{kos} we identify the
modules for the basic algebra that correspond to the Verma modules in
our block, and prove that they are Koszul.

\section{Notation and definitions} \label{notation_section}
\subsection{Algebras}\label{algebra_defns} Let $\mathbb{F}$ be a field.  The Lie algebra $\ssl_3(\mathbb{F})$ is generated by elements $E_i, F_i, H_i$
for $i=1,2$ subject to the relations

\begin{center}
  \begin{tabular}{lll}
     $[H_1,H_2]=0$ &  $[H_i,E_i]=2E_i$ & 
    $[H_i,F_i]=-2F_i$ \\
   $[H_i,E_j]=-E_j$ &  $[H_i,F_j]=F_j$ & 
  $[E_i,F_j]=0$ \\
  
  $[E_i,[E_1,E_2]]=0$ & $[F_i,[F_2,F_1]]=0$ & $[E_i,F_i]=H_i$
\end{tabular}
\end{center}
for $i=1,2$ and $j \neq i$.   We write $E_{12}$ for $[E_1,E_2]$ and
$F_{12}$ for $[F_2,F_1]$.

Let $U_{\ZZ}$ be the Kostant $\ZZ$-form of the universal enveloping
algebra of $\ssl_3(\CC)$: this is the subring generated by all
elements of the form $E_i^{(n)}:= E_i^n/n!$ and $F_i^{(n)}:= F_i^n/n!$.
Let $k$ be an algebraically closed field of characteristic $p>3$ and
$D = U_\ZZ \otimes_\ZZ k$ be the hyperalgebra or algebra of
distributions \cite[I.7, II.1.12]{Jantzen_book}.  We will abuse notation by
writing $E_i^{(n)}$ for the element $E_i^{(n)} \otimes 1$ of $D$ and so
on.  The restricted enveloping algebra $u$ of $\ssl_3(k)$ is defined to
be the subalgebra of $D$ generated by the $E_i$ and $F_i$ for $i=1,2$,
it is a
Hopf algebra of dimension $p^8$
isomorphic to the quotient of the universal enveloping algebra
$U(\ssl_3(k))$ by the two-sided ideal generated by $E_i^p, F_i^p,
E_{12}^p, F_{12}^p, H_i^{p}-H_i$ for $i=1,2$.  We write $b^+$ for the
subalgebra of $u$ generated by the $E_i$ and $H_i$ for $i=1,2$, and
$b^-$ for the subalgebra generated by the $F_i$ and $H_i$ for $i=1,2$.

\subsection{Gradings} Let $\Phi= \{ \pm \alpha_1,\pm\alpha_2, \pm(\alpha_1+\alpha_2)\}$ be the $\ssl_3$ root system with simple roots $\alpha_1$ and
$\alpha_2$, and let $\varpi_1$ and $\varpi_2$ be the fundamental
weights. Let $X$ be the weight lattice, the free abelian group on the
$\varpi_i$, which we identify with $\ZZ^2$ by $\varpi_1 = (1,0)$ and
$\varpi_2=(0,1)$.  Then $U_\ZZ$,
$D$, and $u$ are $X$-graded with $E_i$ in degree $\alpha_i$ and
$F_i$ in degree $-\alpha_i$.  We refer to this as the weight
grading and
write $|x| = \alpha$ if $x$ is a homogeneous element of degree $\alpha
\in X$.  

By \cite{Riche}, $u$ admits a $\ZZ_{\geq 0}$-grading with respect to
which it is a Koszul algebra in the sense of \cite{BGS}. We write $u_i$ for the $i$th
graded piece, and refer to this as the $K$-grading on $u$.  Koszulity implies that $u$ is generated as an algebra by
$u_0$ and $u_1$, and that $u_{>0} = \bigoplus_{i>0}u_i$ is the
Jacobson radical $J$ of $u$.

\subsection{Modules} 
If $M=\bigoplus _{i \in \ZZ} M_i$ is a $K$-graded module, we write
$M[i]$ for the $K$-graded module with the same underlying vector space
and action as $M$, but with grading determined by $M[i]_j = M_{i-j}$.
We write $\Hom_u^K$ and $\Ext_u^{K}$ for the Hom and Ext functors in the
category of finite-dimensional $K$-graded modules.
If $M$ and $N$ are finite-dimensional $K$-graded $u$-modules
\begin{equation}\label{Kdecomp}
	\Hom_u(M,N) = \bigoplus_{i\in \ZZ} \Hom_u^{K}(M,N[i])
\end{equation}
with a similar result for Ext.
We say that a $u$-module $M$ is \textbf{$K$-gradable} if it admits a
decomposition $M = \bigoplus_{i \in \ZZ} M_i$ as vector spaces such that
$u_i M_j \subseteq M_{i+j}$ for all $i$ and $j$.  The simple $u$-modules
are $K$-gradable since $u_{>0}$ acts as zero on any simple, and the
projective $u$ modules are $K$-gradable \cite{Gordon}.

Weight-graded $u$-modules are the same thing as $G_1T$-modules in the
sense of \cite[II.9]{Jantzen_book}. We thus write $\Hom_{G_1T}$ and
$\Ext_{G_1T}$ for the Hom and Ext functors in the category of
weight-graded $u$-modules. 

For each $\alpha \in X$ there is a corresponding one-dimensional
$B_1T$-module $k_\alpha$.  The Verma module $Z(\alpha)$ is defined to be
$G_1T$ module induced from $k_\alpha$, and $L(\alpha)$ denotes the
simple top of $Z(\alpha)$.  The modules $L(a,b)$ with $0\leq a,b<p$
restricted to $u$ form
a complete set of representatives for the simple $u$-modules.  The dual
Verma module $Z'(\alpha)$ is the $G_1T$ module coinduced from
$k_\alpha$; it has simple socle $L(\alpha)$.  

\section{Verma modules} \label{verma_section}
\subsection{Structure of Verma modules} 

For each $0 \leq a \leq p-2$ there is a non-semisimple singular block of
$u$ containing the simple modules $L(p-1,a), L(a,p-2-a)$ and $L(p-2-a,
p-1)$.  From now on we fix such an $a$ and write \[\mu=(p-1,a),\,\,\,
\lambda=(a,p-2-a),\,\,\, \gamma=(p-2-a,p-1).\]  The submodule lattices
for the Verma modules corresponding to these weights were determined in
\cite[Theorems 2.4, 2.5, 2.6]{Xi} and \cite[Theorems 3.3, 3.4]{Irving}.
The following diagrams display these submodule lattices; beneath each
composition factor $L$ is an element of the Verma module which is a
highest weight vector for $L$ modulo the sum of the submodules below it
in the lattice.
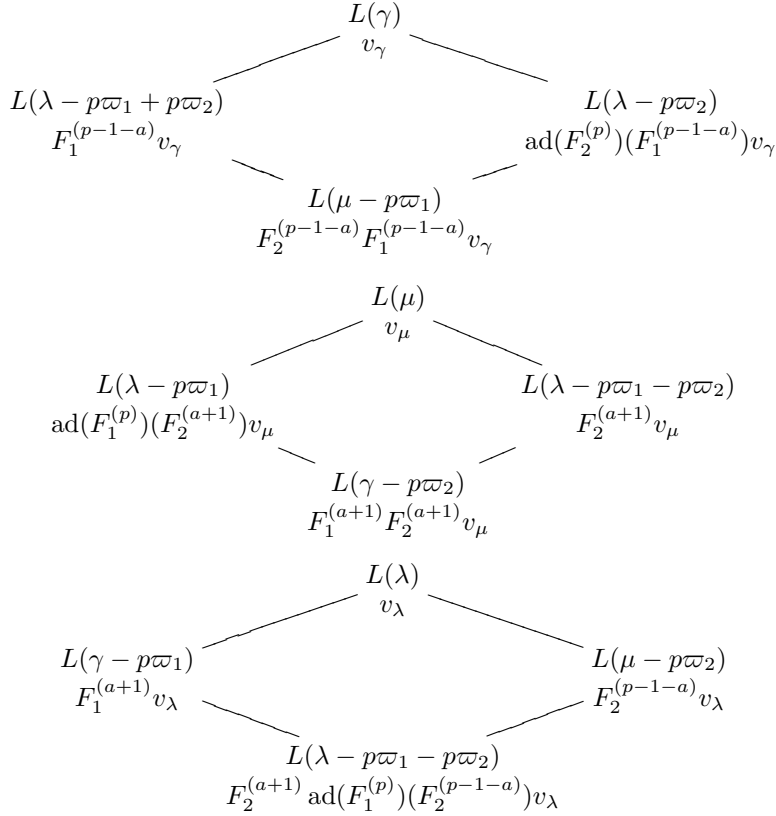
\[
  \xymatrixrowsep{0.2cm} \xymatrixcolsep{0.2cm}
  \xymatrix{
    & \txt{$L(\gamma)$\\$v_\gamma$}   \ar@{-}[dr] \ar@{-}[dl]& \\
     \txt{$L(\lambda-p\vl+p\vt)$\\$F_1^{(p-1-a)}v_\gamma$} \ar@{-}[dr] & & \ar@{-}[dl]
    \txt{$L(\lambda-p\vt)$\\$\ad(F_2^{(p)})(F_1^{(p-1-a)})v_\gamma$}\\
    & \txt{$L(\mu-p\vl)$ \\ $F_2^{(p-1-a)}F_1^{(p-1-a)}v_\gamma$} & } \]
\[
  \xymatrixrowsep{0.2cm} \xymatrixcolsep{0.2cm}
  \xymatrix{
    & \txt{$L(\mu)$\\$v_\mu$} \ar@{-}[dr] \ar@{-}[dl]& \\
    \txt{$L(\lambda-p\vl)$\\$\ad(F_1^{(p)})(F_2^{(a+1)})v_\mu$} \ar@{-}[dr]
    & & \ar@{-}[dl] \txt{$L(\lambda-p\vl-p\vt)$\\$F_2^{(a+1)}v_\mu$} \\
    & \txt{$L(\gamma-p\vt)$\\$F_1^{(a+1)}F_2^{(a+1)} v_\mu$} & } \]
\[
  \xymatrixrowsep{0.2cm} \xymatrixcolsep{0.2cm}
  \xymatrix{
    & \txt{$L(\lambda)$\\$v_\lambda$} \ar@{-}[dr] \ar@{-}[dl]& \\
    \txt{$L(\gamma-p\vl)$\\$F_1^{(a+1)}v_\lambda$} \ar@{-}[dr] & & \ar@{-}[dl]
    \txt{$L(\mu-p\vt)$\\$F_2^{(p-1-a)}v_\lambda$} \\
    &\txt{$
      L(\lambda-p\vl-p\vt)$\\$F_2^{(a+1)}\ad(F_1^{(p)})(F_2^{(p-1-a)})v_\lambda$} & } \]

The structure of the dual Verma modules
follows from these by duality using \cite[Lemma II.9.2]{Jantzen_book}.

\subsection{Verma modules are $K$-gradable}

 \begin{lem} \label{ext1}
  $\Ext^1_u(L(\gamma),L(\mu))=\Ext^1_u(L(\mu),L(\gamma))=\Ext^{1}_u(L(\lambda),L(\lambda))=0$.
\end{lem}
\begin{proof}
	These Ext groups are $G=SL_3(k)$-modules with $G_1$ acting
	trivially.  Suppose for contradiction that
	$\Ext^1_u(L(\gamma),L(\mu))\neq 0$. By the inequality of
	\cite[Lemma 5.1]{Andersen}, as a $G$-module its composition
	factors are isomorphic to $L(0,p)$. So there is an
	indecomposable $G_1T$-module with top $L(\gamma)$ and socle
	$L(\mu-p\varpi_1)$.  Since $\gamma + \alpha_i$ is not in the
	support of $L(\mu-p\varpi_1)$ on the weight lattice for $i=1,2$,
	this module is highest weight, and thus a quotient of
	$Z(\gamma)$. This is impossible from the structure of
	$Z(\gamma)$ given above. Similarly the second Ext group
	vanishes,  and the last is zero by \cite[Proposition
	II.12.9]{Jantzen_book}.
\end{proof}

\begin{prop}
  Let $\alpha \in \{\lambda,\gamma,\mu\}$.  Then $Z(\alpha)$ and
  $Z'(\alpha)$ are
  $K$-gradable.
\end{prop}

\begin{proof}
We give the proof for $Z(\alpha)$, the proof for $Z'(\alpha)$ being
analogous.  The $u$-projective cover $P(\alpha)$ of $L(\alpha)$ is $K$-graded, generated in degree
zero.  Choose a surjection $F:P(\alpha) \to Z(\alpha)$ and let
$Z(\alpha)_0 =
F(P(\alpha)_0)$, a
$u_0$-submodule of $Z(\alpha)$.  Let $Z(\alpha)_2 = \soc_u(
Z(\alpha))$.

The quotient map $\rad Z(\alpha) \to \rad Z(\alpha) / \soc Z(\alpha)$ is
split as a map of $u_0$-modules because $u_0$ is semisimple; let
$Z(\alpha)_1 \subset Z(\alpha)$ be the image of a splitting map.  As
$u_0$-modules, 
\[
	Z(\alpha)|_{u_0} = Z(\alpha)_0 \oplus Z(\alpha)_1
\oplus Z(\alpha)_2
\]
and $Z(\alpha)_1$ is isomorphic as a $u_0$-module to the restriction of
the two middle composition factors of $Z(\alpha)$ to $u_0$.  We claim
that this is a $K$-grading on $Z(\alpha)$.  Since we had a decomposition
of $u_0$-modules, the only thing we need check is that $u_1 Z(\alpha)_0
\subseteq Z(\alpha)_1$.

Let $l_0 \in Z(\alpha)_0$ and $x \in u_1$ and suppose $xl_0=l_1+l_2$ for $l_i \in
Z(\alpha)_i$.  We will show $l_2=0$.  Suppose not, let $\soc Z(\alpha)
\cong
L(\pi)$, and
write $u_0 = \bigoplus _\beta M_\beta$ where $M_\beta $ is a matrix
algebra over $k$ corresponding to the simple module $L(\beta)$.
Pick $q \in
M_\pi$ such that $q l_2=l_2$.  By Lemma~\ref{ext1},
$JP(\alpha)/J^2P(\alpha)$ has no $L(\pi)$ summands, so $M_\pi$ acts as
zero on $P(\alpha)_1=u_1P(\alpha)_0$ and on the $u_0$-module $Z(\alpha)_1$.
Let $\hat l_0 \in P(\alpha)_0$ be such that $F(\hat l_0) = l_0$.  Then
$F(qx \hat l_0) =l_2\neq 0$ so $q\cdot x \hat l_0 \neq 0$, contradicting
that $M_\pi$ acts as zero on $u_1 P(\alpha)_0$.  \end{proof}

Henceforth $Z(\alpha)$ will have the $K$-grading constructed above, so that
its top composition factor is in degree zero, and $Z'(\alpha)$ will be
given the $K$-grading so that its socle $L(\alpha)$ is in degree zero and
its head in degree $-2$.
Thus the multiplicity of $L(\alpha)[i]$
as a composition factor of $Z(\beta)$ equals the multiplicity of
$L(\alpha)[-i]$ as a composition factor of $Z'(\beta)$.

\section{Projective modules and Hilbert series} \label{proj_sec}
In this section we study the structure of the indecomposable
projective modules in our block.

\subsection{$K$-graded Verma filtrations}
We will determine the composition factors of these projectives in the
$K$-graded module category using a Brauer-type reciprocity result.  The
results of  \cite{Holmes} cannot be used directly since $b$ is not a
$K$-graded subalgebra of $u$, however the proofs of the reciprocity
results follow the arguments there closely.

%degree zero part of M is at degree i in the shifted module M[i]

\begin{defn} A $K$-graded module $M$ is said to have a $K$-graded Verma filtration if there
is a sequence $0 = M_0 \subset M_1 \subset \cdots \subset M_s =
M$ of $K$-graded submodules with each quotient $M_r/M_{r-1}$
isomorphic as a $K$-graded module to $Z(\alpha_r)[n_r]$ for some $\alpha_r \in X$ and some $n_r
\in \ZZ$.\end{defn}
Write
$[M : Z(\alpha_r)[n_r]]$ for the number of $M_r/M_{r-1}$ which are
isomorphic as $K$-graded modules to $Z(\alpha_r)[n_r]$
(it is not clear yet that this is
independent of the filtration chosen), and $[M:L(\alpha)[n]]$ for the
multiplicity of $L(\alpha)[n]$ as a composition factor of $M$ in the
category of $K$-graded modules.

\begin{lem}
  Let $\alpha \in \{\lambda, \mu, \gamma\}$.  Then  $P(\alpha)$ has a
  $K$-graded Verma filtration.\end{lem}

\begin{proof}
 More generally, suppose that $M$ is any $u$-module in our block which is
 projective on restriction to $b^-$.  By \cite[Theorem 4.4]{Holmes}
 there is an ungraded injection $\iota: Z(\beta) \to M$ for some $\beta
 \in \{\lambda, \mu, \gamma\}$.
 We may decompose $\iota$ as a sum of $K$-graded
 homomorphisms
 $\iota_j$.  A homomorphism $Z(\beta) \to M$ is injective if and
 only if it is non-zero
 on the simple socle of $Z(\beta)$, and not all the $\iota_j$ can be
 zero on the socle since $\iota$ isn't.  Thus one of the $\iota_j$ is a
 $K$-graded injective module homomorphism $Z(\beta) \to M$.  The
 result follows by induction on the dimension of $M$.
\end{proof}

\begin{prop}
  Let $\alpha, \beta\in \{\lambda, \gamma, \mu\}$ and $i \in \ZZ$.   Then
  \[ [P(\alpha) : Z(\beta)[i]] = [Z(\beta)[-i] : L(\alpha)]. \]
\end{prop}

\begin{proof}
  We first show for any $K$-graded module $M$ with a $K$-graded Verma
  filtration that
  \begin{equation}\label{rec} [M : Z(\beta)[i] ] = \dim \Hom_u^K (M,
      Z'(\beta)[i] ). \end{equation}
The proof is by induction on the length of the Verma filtration.
By \cite[Lemma 4.2]{Holmes} or \cite[II.9.9]{Jantzen_book}, $\Ext^{n}_u(Z(\alpha),Z'(\beta))$ is
isomorphic to $k$ if $n=0$ and $\alpha=\beta$ and is zero otherwise.
Using the Ext version of (\ref{Kdecomp}), if $n\neq 0$ or
$\alpha \neq \beta$
then
$\Ext^{K,n}_u(Z(\alpha),Z'(\beta)[i])=0$, and there is a unique
$i$ such that
$\Hom^K_u(Z(\alpha), Z'(\alpha)[i])$ is nonzero.  Since $Z(\alpha)$ has
simple top $L(\alpha)$ in $K$-degree zero and $Z'(\alpha)$ has simple
socle $L(\alpha)$ in $K$-degree zero, this is $i=0$.
This does the case where the Verma filtration has length one.

Now let
\[ 0 \to X \to M \to Z(\alpha)[r] \to 0 \]
be an exact sequence of graded modules, where $X$ has a $K$-graded
Verma filtration.  For any $\beta, j$ we can apply $\Hom_u^K( -,
Z'(\beta)[j])$ and get a long exact sequence in which the Ext groups
vanish by \cite[Lemma 4.2]{Holmes} again.  So
\[ 0 \to \Hom^K_u(Z(\alpha)[r], Z'(\beta)[j]) \to \Hom^K_u(M,
Z'(\beta)[j]) \to \Hom^K_u(X, Z'(\beta)[j]) \to 0 \]
is exact.
$[M: Z(\beta)[i]]$ is the same as $[X: Z(\beta)[i]]$ except if $\beta =
\alpha, i=r$ when it is one larger.  Equation (\ref{rec}) follows by induction.

To conclude, take $M = P(\alpha)$.  Then $\dim \Hom^K_u(P(\alpha), Z'(\beta)[i])$ is
the multiplicity of $L(\alpha)$ as a composition factor of
$Z'(\beta)[i]$, which is equal to the multiplicity of $L(\alpha)$ as a
composition factor of $Z(\beta)[-i]$ by the remark at the end of the
previous section.
\end{proof}
Applying this result and the structure of the Vermas given previously
gives the factors in a $K$-graded Verma filtration of our projectives.
\begin{center} \begin{tabular}{l|l}
    Projective & $K$-graded Verma factors\\
    \hline
    $P(\lambda)$ &$Z(\lambda)$, $Z(\mu)[1]$,
    $Z(\mu)[1]$,$Z(\gamma)[1]$,
    $Z(\gamma)[1]$, $Z(\lambda)[2]$
    \\
    $P(\gamma)$ &$Z(\gamma)$, $Z(\lambda)[1]$, $Z(\mu)[2]$ \\
    $P(\mu)$ &$Z(\mu)$,$Z(\lambda)[1]$,  $Z(\gamma)[2]$
  \end{tabular} \end{center}

\subsection{Hilbert series and the basic algebra}

\begin{defn} The Hilbert series of our block is the $3\times 3$
matrix $P(t)$ with rows and columns labelled by $\gamma, \lambda, \mu$
whose $\alpha,\beta$ entry is $\sum_i t^i \dim \Hom_u^K (P(\alpha)[i],
P(\beta))$.
\end{defn}

Equivalently, the $\alpha,\beta$ entry of $P(t)$ is the
power series in $t$ whose coefficient of $t^i$ is the multiplicity of
$L(\alpha)[i]$ as a composition factor of $P(\beta)$.  The $K$-graded structure of
the Verma modules determined in Section \ref{verma_section} combined with the table
above immediately give:
\begin{thm}\label{hilb}  The Hilbert series of our block is \[ P(t) = \begin{pmatrix}
    1+t^2+t^4 & 3t+3t^3 & 3t^2 \\
    3t+3t^3 & 1+10t^2+t^4 & 3t+3t^3 \\
    3t^2 & 3t+3t^3 & 1+t^2+t^4 \end{pmatrix}. \]
\end{thm}
The determinant of $P(t)$ is $(t-1)^6(t+1)^6$.
We record the following corollary on the graded structure of the
projectives for future use:

\begin{cor}\label{first_second}
The first and second $K$-graded pieces of the projectives in our block are
as follows.
\begin{center}
  \begin{tabular}{l|lll}
    $\alpha$ & $\lambda$ & $\gamma$ & $\mu$ \\
    \hline
    $P(\alpha)_1$ & $L(\gamma)^{\oplus 3} \oplus L(\mu)^{\oplus 3}$ &
    $L(\lambda)^{\oplus 3}$ & $L(\lambda)^{\oplus 3}$ \\
    $P(\alpha)_2$ & $L(\lambda)^{\oplus 10}$  & $L(\gamma) \oplus
    L(\mu)^{\oplus 3}$ & $L(\mu)\oplus L(\gamma)^{\oplus 3}$
  \end{tabular}
\end{center}
\end{cor}

\begin{defn} The basic algebra of our block is $\Gamma = \End_u
  (P(\mu)\oplus P(\gamma)\oplus P(\lambda))^{\text{op}}$.\end{defn}
Since $\Gamma$ is Morita equivalent \cite[\S 2.2]{BensonI} to a block of the
Koszul algebra $u$, it is also Koszul with respect to the grading
arising from the $K$-grading on the projectives \cite[F.3]{AJS}.  
Note that the dimension of the $n$th $K$-graded part
$\Gamma_n$ is $3 {4 \choose n}$; compare Manin's quantum Grassmann
algebras \cite[\S 8]{manin}.

\begin{thm}
The Hilbert series of the Koszul dual \[\Gamma^! \cong \Ext_u^*(L(\mu)\oplus L(\gamma)\oplus L(\lambda),
L(\mu)\oplus L(\gamma)\oplus L(\lambda) )\] is:
\[  (t-1)^{-4}(t+1)^{-4}\begin{pmatrix}
1+4t^2+t^4 & 3t(t^2+1) & 6t^2 \\
3t(t^2+1) & 1+4t^2 + t^4 & 3t(t^2+1) \\
6t^2 & 3t(t^2+1) & 1+4t^2 + t^4 \end{pmatrix}. \]
\end{thm}
% No sign errors, checked on Maxima
\begin{proof}
By \cite[Lemma 2.11]{BGS} the Hilbert series of $\Gamma^!$ is $P(-t)^{-T}$,
so the result follows from Theorem \ref{hilb}. \end{proof}
As power series, the diagonal entries are $\sum_k (k+1)^3 t^{2k}$ and
the off-diagonal entries are $\frac{1}{2}\sum_k
(k+1)(k+2)(2k+3)t^{2k+1}$ and $\sum_k k(k+1)(k+2) t^{2k}$.

\subsection{Weight-graded Verma filtrations}
Each projective has a filtration by $G_1T$ Verma
modules, and we write
$[P(\alpha):Z(\beta)]$ for the number of factors in a
Verma filtration of $P(\alpha)$ isomorphic as $G_1T$-modules 
to $Z(\beta)$.
In \cite[Theorem 5.1]{Holmes} the authors prove a result determining
these multiplicities
in a category of graded $u$-modules, but it is stated only
for a $\ZZ$-grading obtained by flattening the weight grading. Their
proofs go through in the weight-graded category however, or we can
instead apply the $n=1$ case of \cite[Satz 3.8]{Jantzen_uber}, to get:

\begin{prop}
  Let $\alpha, \beta \in X$. Then
  \[  [P(\alpha):Z(\beta)] =
  [Z(\beta):L(\alpha)]. \]
\end{prop}
Again, using the structure of the Vermas given earlier gives:
\begin{center}\begin{tabular}{l|l}
  Projective & Weight-graded Verma factors \\ \hline
  $P(\gamma)$ & $Z(\gamma), Z(\lambda+p\vl),
Z(\mu+p\vt)$ \\
  $P(\lambda)$ & $Z(\lambda),
Z(\mu+p\vl), Z(\mu-p\vl+p\vt), Z(\gamma+p\vl-p\vt),
Z(\gamma-p\vt), Z(\lambda+p\vl+p\vt)$
  \\
  $P(\mu)$ & $Z(\mu), Z(\lambda+p\vt), Z(\gamma+p\vl).$
\end{tabular}\end{center}
This determines the composition series of the projectives 
as $G_1T$-modules.  Figures~\ref{Pgamma}, \ref{Pmu} and
\ref{Plambda} display these structures, where a double circle represents
a composition factor of multiplicity two and so on, and we have written
$\alpha$ instead of $L(\alpha)$.  The coloured diamonds join the
composition factors of a term in the Verma filtration of $P(\alpha)$,
blue, orange and teal diamonds are shifts of $Z(\lambda)$, $Z(\gamma)$ and
$Z(\mu)$ respectively.
\begin{figure}
\begin{tikzpicture}
\node[fill=black,circle,inner sep=1pt,label=above:$\gamma$] at (0,0) {}; 
\node[fill=black,circle,inner sep=1pt,label=above:{$\mu+p\vt$}] at (0,1) {};
\node[fill=black,circle,inner sep=1pt,label=below:{$\lambda-p\vt$}] at (0,-1) {};
\node[fill=black,circle,inner sep=1pt,label=right:{$\lambda+p\vl$}] at ({0.5*sqrt(3)}, 0.5) {};
\node[fill=black,circle,inner sep=1pt,label=left:{$\mu-p\vl$}] at ({-0.5*sqrt(3)}, -0.5) {};
\node[fill=black,circle,inner sep=1pt,label=right:{$\mu+p\vl-p\vt$}] at ({0.5*sqrt(3)}, -0.5) {};
\node[fill=black,circle,inner sep=1pt,label=left:{$\lambda-p\vl+p\vt$}] at ({-0.5*sqrt(3)}, 0.5) {};
\draw [gray] (0,0) circle (3pt);
\draw [gray] (0,0) circle (5pt);
\draw [gray] ({0.5*sqrt(3)}, 0.5) circle (3pt);
\draw [gray] (0, -1) circle (3pt);
\draw [gray] ({-0.5*sqrt(3)}, 0.5) circle (3pt);
\draw[color=blue]
(0.02,-0.02) -- (0.02,-1) -- ({0.5*sqrt(3)},-0.5) -- ({0.5*sqrt(3)},0.48) -- cycle;
\draw[color=orange]
(-0.02,-0.02) -- ({-0.5*sqrt(3)}, 0.48) -- ({-0.5*sqrt(3)},-0.5) --
(-0.02,-1) -- cycle;
\draw[color=teal]
(0,0.02) -- ({0.5*sqrt(3)},0.52) -- (0,1) -- ({-0.5*sqrt(3)},0.52) -- cycle;
\end{tikzpicture}
\caption{$P(\gamma)$} \label{Pgamma} \end{figure}
\begin{figure}
\begin{tikzpicture}
  \node[fill=black,circle,inner sep=1pt,label=below:$\mu$] at (0,0) {}; 
  \node[fill=black,circle,inner sep=1pt,label=above:{$\lambda+
  p\vt$}] at (0,1) {};
  \node[fill=black,circle,inner sep=1pt,label=below:{$\gamma-p\vt$}] at (0,-1) {};
  \node[fill=black,circle,inner sep=1pt,label=right:{$\gamma+p\vl$}] at ({0.5*sqrt(3)}, 0.5) {};
  \node[fill=black,circle,inner sep=1pt,label=left:{$\lambda-p\vl$}] at ({-0.5*sqrt(3)}, -0.5) {};
  \node[fill=black,circle,inner sep=1pt,label=right:{$\lambda+p\vl-p\vt$}] at ({0.5*sqrt(3)}, -0.5) {};
  \node[fill=black,circle,inner sep=1pt,label=left:{$\gamma-p\vl+p\vt$}] at ({-0.5*sqrt(3)}, 0.5) {};
\draw [gray] (0,0) circle (3pt);
\draw [gray] (0,0) circle (5pt);
\draw [gray] ({0.5*sqrt(3)}, -0.5) circle (3pt);
\draw [gray] (0, 1) circle (3pt);
\draw [gray] ({-0.5*sqrt(3)}, -0.5) circle (3pt);
\draw [color=blue]
(-0.02,0.02)--(-0.02,1)--({-0.5*sqrt(3)},0.5)--({-0.5*sqrt(3)},-0.48)--cycle;
\draw [color=orange]
(0.02,
0.02)--(0.02,1)--({0.5*sqrt(3)},0.5)--({0.5*sqrt(3)},-0.48)--cycle;
\draw[color=teal]
(0,-0.02)--({0.5*sqrt(3)},-0.52)--(0,-1)--({-0.5*sqrt(3)},-0.52)--cycle;
\end{tikzpicture}
\caption{$P(\mu)$} \label{Pmu} \end{figure}
\begin{figure}
\begin{tikzpicture}
  \node[fill=black,circle,inner sep=1pt,label={[label
  distance=0.1cm]180:{$\lambda$}}]  at (0,0) {}; 
  \node[fill=black,circle,inner sep=1pt,label={[label
  distance=0.18cm]90:{$\gamma+p\vt$}}] at (0,1) {};
  \node[fill=black,circle,inner sep=1pt,label={[label
  distance=0.15cm]270:{$\mu-p\vt$}}] at (0,-1) {};
  \node[fill=black,circle,inner sep=1pt,label={[label
  distance=0cm]10:{$\mu+p\vl$}}] at ({0.5*sqrt(3)}, 0.5) {};
  \node[fill=black,circle,inner sep=1pt,label={[label
  distance=0cm]250:{$\gamma-p\vl$}}] at ({-0.5*sqrt(3)}, -0.5) {};
  \node[fill=black,circle,inner sep=1pt,label=right:{$\gamma+p\vl-p\vt$}] at ({0.5*sqrt(3)}, -0.5) {};
  \node[fill=black,circle,inner sep=1pt,label={[label
  distance=0cm]140:{$\mu-p\vl+p\vt$}}] at ({-0.5*sqrt(3)}, 0.5) {};
  \node[fill=black,circle,inner sep=1pt,label=right:{$\lambda+p\alpha_1$}] at ({sqrt(3)},0) {};
  \node[fill=black,circle,inner sep=1pt,label=left:{$\lambda-p\alpha_1$}] at ({-sqrt(3)},0) {};
  \node[fill=black,circle,inner
  sep=1pt,label=right:{$\lambda+p\alpha_1+p\alpha_2$}] at ({0.5*sqrt(3)}, 1.5) {};
  \node[fill=black,circle,inner
	  sep=1pt,label=left:{$\lambda-p\alpha_1-p\alpha_2$}] at ({-0.5*sqrt(3)}, -1.5) {};
  \node[fill=black,circle,inner sep=1pt,label=right:{$\lambda-p\alpha_2$}] at ({0.5*sqrt(3)}, -1.5) {};
  \node[fill=black,circle,inner sep=1pt,label=left:{$\lambda+p\alpha_2$}] at ({-0.5*sqrt(3)}, 1.5) {};
\draw [gray] (0,0) circle (2pt);
\draw [gray] (0,0) circle (3pt);
\draw [gray] (0,0) circle (4pt);
\draw [gray] (0,0) circle (5pt);
\draw [gray] (0,0) circle (6pt);
\draw [gray] ({0.5*sqrt(3)}, -0.5) circle (3pt);
\draw [gray] (0, 1) circle (3pt);
\draw [gray] ({-0.5*sqrt(3)}, -0.5) circle (3pt);
\draw [gray] ({0.5*sqrt(3)}, 0.5) circle (3pt);
\draw [gray] (0, -1) circle (3pt);
\draw [gray] ({-0.5*sqrt(3)}, 0.5) circle (3pt);
\pgfmathsetmacro{\EEE}{0.04}
\draw [color=blue]
($(0,0)+\EEE/3*({0.5*sqrt(3)},1.5)$)--($(0,1)+\EEE/3*({0.5*sqrt(3)},1.5)$)
-- ({0.5*sqrt(3)},1.5) --
($({0.5*sqrt(3)},0.5)+\EEE/3*({0.5*sqrt(3)},1.5)$) -- cycle;
\draw[color=teal]
($(0,0)+\EEE*(1,0)$)--($({0.5*sqrt(3)},0.5)+\EEE*(1,0)$)
-- ({sqrt(3)},0) --
($({0.5*sqrt(3)},-0.5)+\EEE*(1,0)$) -- cycle;
\draw[color=orange]
($(0,0)+\EEE/3*({0.5*sqrt(3)},-1.5)$)--($(0,-1)+\EEE/3*({0.5*sqrt(3)},-1.5)$)
-- ({0.5*sqrt(3)},-1.5) --
($({0.5*sqrt(3)},-0.5)+\EEE/3*({0.5*sqrt(3)},-1.5)$) -- cycle;
\draw[color=blue]
($(0,0)-\EEE/3*({0.5*sqrt(3)},1.5)$)--($(0,-1)-\EEE/3*({0.5*sqrt(3)},1.5)$)
-- ({-0.5*sqrt(3)},-1.5) --
($({-0.5*sqrt(3)},-0.5)-\EEE/3*({0.5*sqrt(3)},1.5)$) -- cycle;
\draw[color=teal]
($(0,0)-\EEE*(1,0)$)--($({-0.5*sqrt(3)},-0.5)-\EEE*(1,0)$)
-- ({-sqrt(3)},0) --
($({-0.5*sqrt(3)},0.5)-\EEE*(1,0)$) -- cycle;
\draw[color=orange]
($(0,0)-\EEE/3*({0.5*sqrt(3)},-1.5)$)--($(0,1)-\EEE/3*({0.5*sqrt(3)},-1.5)$)
-- ({-0.5*sqrt(3)},1.5) --
($({-0.5*sqrt(3)},0.5)-\EEE/3*({0.5*sqrt(3)},-1.5)$) -- cycle;
\end{tikzpicture}
\caption{$P(\lambda)$}\label{Plambda}\end{figure}
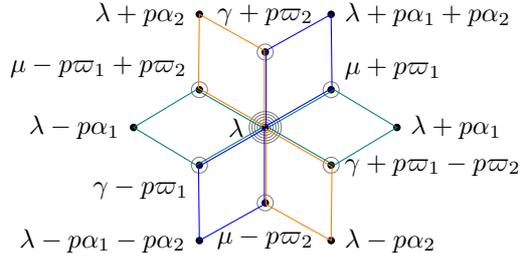

\section{$G$-module structure}\label{sec:action}

Let $G=SL_3(k)$.  The Ext-algebra
\begin{equation*}
E=	\Ext^*_u (L(\lambda)\oplus L(\gamma)\oplus L(\mu),
	L(\lambda)\oplus L(\gamma)\oplus L(\mu))
\end{equation*}
is a $G$-module, as in \cite{Andersen}, with $G_1$ acting trivially.
We can determine the $G$-module structure of the degree one part of
this algebra:

\begin{lem}
As $G$-modules,
\begin{align*}
	\Ext^1_u(L(\lambda),L(\mu)) \cong \Ext^1_u(L(\gamma),L(\lambda))
	& \cong L(0,p) \\
\Ext^1_u(L(\mu),L(\lambda)) \cong \Ext^1_u(L(\lambda),L(\gamma))
	& \cong L(p,0).
\end{align*}
\end{lem}
\begin{proof}
	The dimensions of these Ext-groups are determined by Corollary
	\ref{first_second}. The $G$-module structure follows from the
	inequality of \cite[Lemma 5.1]{Andersen}. Alternatively, we
	can read off some of the $T$-weights of these modules from the
	extensions given by quotients of the Verma modules displayed in
	Section \ref{verma_section}, which together with the dimensions
	are enough to determine the $G$-module structure.
\end{proof}

\section{Presentation of the basic algebra} \label{sec:pres}

In this section we consider the problem of finding generators and
relations for the basic algebra of our block. Because this algebra is
Koszul, we instead solve the equivalent problem of finding an algebra
whose quadratic
dual is the Ext-algebra $E$ of the previous section: this algebra will
be the opposite algebra of the basic algebra we want.

That is, letting
$\mathbbm{k} = E_0 = k\langle e_\lambda, e_\gamma, e_\mu\rangle$
where $e_\alpha$ is the identity map on $P(\alpha)$ and zero on the
other summands, we look for a $\mathbbm{k}$-bimodule $V$ and subspace $R
\subset V\otimes_\mathbbm{k}V$ such that if $\Lambda =
T_\mathbbm{k}(V)/(R)$ then $\Lambda^!=T_\mathbbm{k}(V^*)/(R^\perp)\cong
E$.

Since $G$ acts on $E$, we can take $V$ to be a $G$-module such that the
dual action on $V^*$ agrees with the action of $G$ on the degree one
part of $E$ which can be identified with $V^*$.  So
\begin{align*}
	e_\lambda V e_\mu \cong e_\gamma V e_\lambda & \cong L(0,p) \\
	e_\mu V e_\lambda \cong e_\lambda V e_\gamma & \cong L(p,0)
\end{align*}
and $e_\alpha Ve_\alpha = e_\gamma V e_\mu = e_\mu V e_\gamma = 0$ for
all $\alpha = \lambda,\gamma,\mu$.
Since the action of $G$ is by algebra automorphisms,
$R^\perp$ and $R$ will be $G$-submodules of
$V^*\otimes_\mathbbm{k}V^*$ and $V\otimes_\mathbbm{k}V$ respectively.

The $\alpha, \beta$ entry of the following table records the
$G$-module structure on $e_\alpha V\otimes_\mathbbm{k} V e_\beta$:

\begin{center}
\begin{tabular}{c|ccc}
  & $\gamma$ & $\lambda$ & $\mu$ \\
  \hline
 $\gamma$& $L(0,0)\oplus L(p,p)$ & &$L(0,p)\oplus L(2p,0)$ \\
 $\lambda$ & & $L(0,0)^{\oplus 2} \oplus L(p,p)^{\oplus 2}$ & \\
 $\mu$ &$ L(p,0)\oplus L(0,2p)$ & & $L(0,0) \oplus L(p,p)$ \end{tabular}
\end{center}

Write $v_{p,0}$, $w_{p,0}$ for highest weight vectors of $e_\lambda V
e_\gamma$ and $e_\mu V e_\lambda$, and $v_{0,p}, w_{0,p}$ for highest
weight vectors of $e_\gamma V e_\lambda$ and $e_\lambda V e_\mu$
respectively.

Corollary \ref{first_second} tells us the dimensions of each $e_\alpha
\Gamma_2 e_\beta$. For example, $e_\gamma \Gamma_2 e_\gamma$ is one-dimensional,
and so $e_\gamma Re_\gamma$ must be a $G$-submodule of
codimension one in $e_\gamma V\otimes_\mathbbm{k} V e_\gamma$, and therefore consists of the $L(p,p)$ summand from
the top row of the
previous table, which is generated as a $G$-module by
$v_{0,p}\otimes v_{p,0}$. Similarly, $R$ contains exactly the $L(2p,0)$
summand on the top row of the above table, and the $L(0,2p)$ and
$L(p,p)$ summands on the bottom row.

We still need to identify $e_\lambda R e_\lambda$, which must be an
eight-dimensional submodule of $e_\lambda V\otimes_\mathbbm{k}
Ve_\lambda$ by Corollary \ref{first_second}.  It is therefore isomorphic
to $L(p,p)$ as a $G$-module, and so is generated as a
$G$-module by some element of
the form $r v_{p,0}\otimes v_{0,p} + s w_{0,p}\otimes w_{p,0}$ for $r,s \in
k$.

Suppose $r=0$.  Then in the quadratic dual algebra $\Lambda^! \cong E$
we have $v_{p,0}^* \cdot v_{0,p}^* = 0$.  Now $v_{p,0}^*$ corresponds to
a weight-graded extension of $L(\lambda)$ by $L(\gamma+p\vl)$ and
$v_{0,p}^*$ corresponds to a weight-graded extension of $L(\gamma)$ by
$L(\lambda+p\vt)$. If this product were zero in the Ext-ring, there
would be a weight-graded uniserial module with top $L(\lambda)$, middle
composition factor $L(\gamma-p\vl)$ and socle $L(\lambda-p\vl-p\vt)$ as in
\cite[Proposition 2.3(a)]{BC}.
Such a module is necessarily highest weight, hence a quotient
of $Z(\lambda)$, but the structure of $Z(\lambda)$ given in Section
\ref{verma_section} shows it has no such quotient. Therefore $r\neq 0$, and similarly $s \neq 0$. By
rescaling $v_{0,p}$ and the $G$-submodule it generates, we may
assume $r=s=1$.
We have proved:

\begin{thm} \label{presentation}
  The quadratic dual of $E$ is isomorphic to $T_\mathbbm{k}(V)/(R)$, where $V$ and $\mathbbm{k}$ are as
  above and $R$ is the $G$-submodule of $V\otimes_\mathbbm{k} V$
  generated by $v_{0,p}\otimes v_{p,0}$,  $w_{p,0}\otimes v_{p,0}$, $w_{p,0}
  \otimes w_{0,p}$, $v_{0,p}\otimes w_{0,p}$, $v_{p,0}\otimes v_{0,p} +
  w_{0,p}\otimes w_{p,0}$.
\end{thm}
The following table displays $R$ as a submodule of
$V\otimes_\mathbbm{k}V$:
\begin{center}
  \begin{tabular}{l|lcl}
    & $\gamma$ & $\lambda$ & $\mu$ \\ \hline
    $\gamma$ & $L(p,p)$ & & $L(2p,0)$ \\
    $\lambda$& & Diagonal $L(p,p)$ & \\
    $\mu$ & $L(0,2p)$ & & $L(p,p)$
  \end{tabular}
\end{center}

In fact $\Lambda^{\text{op}}\cong \Lambda$ in this case,
so $\Lambda$ itself is isomorphic to the basic algebra of our block.

\section{Hochschild cohomology} \label{hoch_section}
It is well-known that any non-semisimple block of $u(\ssl_2(k))$ has as
its basic algebra a
smash product $\Sigma = kC_2 \ltimes k[x,y]/(x^2,y^2)$ where the cyclic
group $C_2=\langle g \rangle$ acts by $gxg=-x$ and $gyg=-y$.  A
straightforward computation with this basic algebra shows 
the centre $Z(\Sigma)$ has dimension 3 and
$\HH^1(\Sigma) \cong \mathfrak{gl}_2(k)$ as a Lie algebra.  In this
section we investigate the centre and Hochschild cohomology of our block
of $u$. Since the block algebra is $K$-graded, its Hochschild cohomology has an
internal grading coming from this $K$-grading. We will calculate the centre
of $\Lambda$ and the internal degree zero part of its first Hochschild
cohomology, showing that as in the $\ssl_2$ case, it is isomorphic to
the general linear Lie algebra.

Let $\Lambda = T_\mathbbm{k}(V)/(R)$ as in the previous section.  Since $G$ acts on $\Lambda$ with
$G_1$ acting trivially, there is an induced action of the Lie algebra $\ssl_3(k)$
(with its weight degrees multiplied by $p$)
on $\Lambda$ by derivations.  Let $v_{-p,0}=F_1F_2v_{0,p}$, $v_{p,-p}=F_2v_{0,p}$,
$v_{-p,p}=-F_1v_{p,0}$, $v_{0,-p}=F_2F_1v_{p,0}$ in $\Lambda$, and make similar
definitions for the $w$s.
The relations for $\Lambda$ are then
\begin{center}
  \begin{tabular}{ll}
    $e_\gamma R e_\gamma$: & $v_\alpha v_{-\beta} = 0$ for  $\alpha \neq
    \beta$, \\
    &  $v_\alpha v_{-\alpha} = v_\beta v_{-\beta}$ for all $\alpha,\beta$. \\
    $e_\mu R e_\mu$: &  $w_{-\alpha} w_\beta=0$ for $\alpha \neq \beta$, \\
    &  $w_{-\alpha} w_{\alpha}=w_{-\beta} w_{\beta}$ for all
    $\alpha,\beta$. \\
    $e_\mu R e_\gamma$: & $w_{-\alpha} v_{-\beta} + w_{-\beta}
    v_{-\alpha} =0$ for all $\alpha,\beta$. \\
$ e_\gamma R e_\mu$: &
 $v_\alpha w_\beta + v_\beta w_\alpha =0$ for all $\alpha,\beta$. \\
$e_\lambda R e_\lambda$: &
$v_{-\alpha} v_\beta + w_\beta w_{-\alpha}$ for   $\beta \neq \alpha$, \\
&
$v_{-\alpha}v_{\alpha}+w_{\alpha}w_{-\alpha}=v_{-\beta}v_{\beta}+w_{\beta}w_{-\beta}$
for all $\alpha,\beta$
 \end{tabular}\end{center}
where $\alpha,\beta \in \{(p,-p), (-p,0), (0,p)\}$.
Let
\begin{center}
\begin{tabular}{llll} $\mathcal{V} = \sum v_\alpha v_{-\alpha}$ &
 $ \mathcal{W} =\sum
w_{-\alpha} w_\alpha $ \\
$z_{\gamma\lambda}  = 3v_{0,p}v_{0,-p}
+\mathcal{V}-2\mathcal{W}$ & $ z_{\mu\lambda} = 3w_{0,-p}w_{0,p} +
\mathcal{W}-2\mathcal{V}$\end{tabular}
\end{center}
where both sums are over $\alpha \in \{(-p,p), (p,0),
(0,-p)\}$.  Let  $z_\alpha$ span the socle of $\Lambda e_\alpha$ for $\alpha =
\lambda, \mu, \gamma$.

\begin{prop}The centre $Z(\Lambda)$ has basis $1, z_\lambda, z_\mu, z_\gamma,
  z_{\mu\lambda}, z_{\gamma\lambda}$.\end{prop}

\begin{proof}
We first show that the given elements are central.   Each of the elements
$z_\alpha$ lie in the centre because they are killed on both sides by
any element of $V$. Using the relations above: 
\begin{center}
  \begin{tabular}{l|llll}
    $v_x$ & $v_x \mathcal{V}$ & $\mathcal{V}v_x$ & $v_x\mathcal{W}$ &
    $\mathcal{W}v_x$ \\
    \hline
    $e_\lambda V e_\gamma$ & 0 & $v_x v_{-x} v_x$ & 0 & $2v_x v_{-x}
    v_x$ \\
    $e_\gamma V e_\lambda$ & $v_x v_{-x}v_x$ & 0 & $2v_xv_{-x}v_x$ & 0
    \\
    \multicolumn{5}{c}{ }  \\
    $w_x$ & $w_x \mathcal{V}$ & $\mathcal{V}w_x$ & $w_x\mathcal{W}$ &
    $\mathcal{W}w_x$ \\
    \hline
    $e_\lambda V e_\mu$ & $2w_xw_{-x}w_x$ & 0& $w_xw_{-x}w_x$ & 0
     \\
     $e_\mu V e_\lambda$ & 0 & $2w_xw_{-x}w_x$ & 0 & $w_xw_{-x}w_x$.
  \end{tabular}
\end{center}
From these it follows easily that $z_{\gamma \lambda}$ and
$z_{\mu\lambda}$ are central.

Any element of the centre lies in $e_\lambda \Lambda e_\lambda +
 e_\gamma \Lambda e_\gamma + e_\mu \Lambda e_\mu$, since this is the condition to
 commute with the idempotents $e_\mu, e_\gamma$ and $e_\lambda$.
Since it is homogeneous with respect to the $K$-grading, it can be
non-zero only in degrees $0,2,4$ by considering the Hilbert series for
$\Lambda$.  Furthermore derivations preserve the centre of an algebra,
so $Z(\Lambda)$ is a
$\ssl_3(k)$-submodule.

From the Hilbert series, $e_\alpha \Lambda_2 e_\alpha$ is
one-dimensional for $\alpha=\mu, \gamma$ and $L(p,p)\oplus
L(0,0)^{\oplus 2}$ as a $G$-module for $\alpha = \lambda$.
The element $v_{p,0}v_{0,p}$ does not commute with $v_{p,0}$, thus the
degree two part of the centre is contained in the trivial summands. No
nonzero element of the two-dimensional subspace spanned by $v_{0,p}v_{0,-p}$ and
$w_{p,0}w_{-p,0}$ commutes with both $v_{0,p}$ and $w_{p,0}$, so the degree two part of the centre has dimension at most two.
It is therefore spanned by
$z_{\gamma\lambda}$ and $z_{\mu\lambda}$.  Any element of $\Lambda_4$ is
in the span of $z_\lambda$, $z_\mu$ and $z_\gamma$, so they certainly
form a basis of the degree four part of the centre.  Lastly a central
element of $\Lambda_0$ must be a scalar multiple of the identity, so we
are done.
\end{proof}

We already know some non-inner $K$-grading-preserving derivations of $\Lambda$, namely those
induced by the action of $\ssl_3(k)$ and the grading derivation $\Delta$
defined by $\Delta(l) = nl$ for $l \in \Lambda_n$.

\begin{prop} The internal degree zero part of $\HH^1(\Lambda)$ is
  isomorphic to $\mathfrak{gl}_3(k)$ and is spanned by the
  image of $\Delta$ together with the images of the derivations arising
  from the $\ssl_3(k)$-action.
\end{prop}
\begin{proof}
  Let $\delta$ be a non-inner derivation of $\Lambda$ which preserves
  the $K$-grading.
  Since $\delta$ preserves $\Lambda_0=\langle
  e_\lambda,e_\gamma,e_\mu\rangle$ and the $e_\alpha$ are idempotent, it
  must act as zero on $\Lambda_0$. Therefore $\delta$
  is determined by its
  action on $\Lambda_1$, and
  must send $e_\alpha\Lambda_1 e_\beta$ to
  $e_\alpha\Lambda e_\beta$ for each $\alpha,\beta$.  

  So $\delta$ is
  determined by the four linear maps $\delta_{\alpha\beta} \in
\End_k(e_\alpha
  \Lambda_1 e_\beta)$, for $(\alpha, \beta) = (\gamma, \lambda), (\lambda,
  \gamma), (\mu,\lambda), (\lambda, \mu)$.  Any such maps extend
  uniquely to a derivation of $T_\mathbbm{k}(V)$.  The condition for
  $\delta$ to be a derivation is that these maps, when extended to a
  derivation on $T_{\mathbbm{k}}(V)$, must preserve $R \subseteq V
  \otimes _{\mathbbm{k}} V$.  Identify $e_\gamma V e_\lambda$ and
  $e_\lambda V e_\mu$ with $L(0,p)$, and $e_\lambda V e_\gamma$ and
  $e_\mu V e_\lambda$ with $L(0,p)^*$.  Then $e_\mu R e_\gamma$ is the
  symmetric square of $L(0,p)$, and if $a$ and $b$ are linear
  endomorphisms of $L(0,p)$ then $1\otimes a + b \otimes 1$ preserves
  the symmetric square of $L(0,p)$ if and only if $a$ and $b$ differ by
  a scalar multiple of the identity.
  Thus $\delta_{\mu \lambda}- \delta_{\lambda \gamma}$ and $\delta_{\gamma \lambda} -
  \delta_{\lambda \mu}$ are scalar multiples of the identity.

  Write $\delta_{\mu \lambda} = rI + x$ and $\delta_{\lambda \mu}=
  sI + y$ where $I$ is the identity, $r,s\in k$ and $x$ and $y$ are the
  linear endomorphisms of $L(0,p)$ and $L(0,p)^*$ induced by the action
  of elements
  $X,Y \in \ssl_3(k)$.  Note that under our
  identifications, $e_\mu R e_\mu$ is
  the kernel of the evaluation map $\nu: L(0,p) \otimes L(0,p)^* \to k$.
  We claim that the endomorphism \[(\delta_{\mu \lambda}-rI)\otimes 1 + 1 \otimes
  (\delta_{\lambda \mu}-sI)=x\otimes 1 + 1 \otimes y\] has image contained in $\ker \nu$.
  Note $L(0,p)\otimes L(0,p)^* = \ker \nu \oplus L(0,0)$ where
  the trivial summand is spanned by $v=\sum v_{ij}\otimes v_{ij}^*$, and
  that $\ker \nu$ belongs to $R$ so is preserved by this endomorphism
  since it differs from the action of $\delta$ by a scalar multiple of
  the identity.  Both $x\otimes 1$ and $1\otimes y$ act on $v$ by scalar
  multiplication by their trace, which is zero, proving the claim.
  
  This shows that for 
  any $u \in L(0,p)$ and $f \in L(0,p)^*$,
  \[ f(x)(u) +
  (y(f))(u) = 0\]
  that is, $(Y\cdot f) (u) = f(-X\cdot u)$ for all $f,u$.  But by
  definition of the action on the dual space, $(Y\cdot f)(u) = f(-Y\cdot
  u)$.  It follows
  $X=Y$.  We now have
  \begin{align*}
  \delta_{\mu \lambda} &= rI + \rho(X) &    \delta_{\lambda \gamma} &=r'I +
  \rho(X) \\ \delta_{\lambda \mu} &=sI + \rho' (X)  &  \delta_{\gamma
  \lambda} & = s'I + \rho' (X) \end{align*} for some $X \in \ssl_3(k)$ and
$r,r',s,s' \in k$
  where $\rho$ and $\rho'$ are the representations corresponding to
  $L(0,p)$ and its dual.
  The only part of $R$ we
  are yet to consider is $e_\lambda R e_\lambda$, but
  $\delta$ preserves this if and only if $r+s=r'+s'$.  Adding
  \[((r+s)/2+r')\ad(e_\gamma) + ( ( s-r)/2)\ad(e_\mu)\] shows
  $\delta$ is $((r+s)/2)\Delta$ plus the action of an element of
  $\ssl_3(k)$.
  
  The inner derivations of degree zero are spanned by the
  $\ad(e_\alpha)$s which are linearly independent of
  $\Delta$ and the $\ssl_3(k)$-action. Any linear dependence could only
  involve the $H_i$ and $\Delta$ action because the $\ad(e_\alpha)$
  preserve the weight grading.  $\Delta$ and the $\ad(e_\alpha)$ act as
  a scalar on $e_\lambda V e_\gamma$, but no nonzero linear combination
  of the $H_i$ does, so the $H_i$ cannot be involved.  Finally $\Delta$
  is not a linear combination of $\ad(e_\gamma)$ and $\ad(e_\mu)$ since
  it acts by the same scalar on $e_\gamma V e_\lambda$ as it does on
  $e_\lambda V e_\mu$.
\end{proof}

\section{Verma modules in this block are Koszul} \label{kos}
Recall that a graded module $M$ over a Koszul algebra
$\Gamma$ is called Koszul if
it has a graded projective resolution
$P^* \twoheadrightarrow M$ such that $P^i = \Gamma P^i_i$. Such a projective resolution
is called linear. Equivalently $M$ is Koszul if $\Ext^{i,j}_\Gamma (M,
\Gamma/\Gamma_{>0})$ is zero unless $i=j$, where $i$ is the
homological degree and $j$ the internal degree arising from the
$\ZZ_{\geq 0}$-grading on $\Gamma$.
Koszulity implies $M
\cong (\Gamma \otimes_\mathbbm{k} M_0) / H$ where $H \subseteq \Gamma_1\otimes M_0$.
In this section we show that the Verma modules $Z(\lambda)$, $Z(\gamma)$
and $Z(\mu)$ are Koszul. 

$\Lambda$ has an $X$-grading which we call the weight grading obtained by putting the $e_\alpha$ in
degree $(0,0)$ and the $v_{i,j}$ and $w_{i,j}$ in degree $(i,j)$.

\begin{lem}
  Under the Morita equivalence between our block of $u$ and $\Lambda$,
  the Verma modules $Z(\lambda), Z(\gamma), Z(\mu)$ correspond to
  \[Z_\lambda=\frac{\Lambda e_\lambda }{
\Lambda v_{0,p}+\Lambda v_{p,-p}+\Lambda w_{-p,p} + \Lambda
w_{p,0}},Z_\gamma=\frac{\Lambda e_\gamma }{ \Lambda v_{p,0}},
Z_\mu=\frac{\Lambda e_\mu}{ \Lambda w_{0,p}} \]
respectively.
\end{lem}

\begin{proof}
Let $S_\alpha$ be the weight-graded simple $\Lambda$-module $\Lambda
e_\alpha/J(\Lambda)e_\alpha$.  The Morita correspondent of $Z(\gamma)$
is a quotient $\Lambda e_\gamma/I_\gamma$ with weight-graded composition
factors $S_\gamma, S_\lambda[-p,p], S_\lambda[0,-p], S_\mu[-p,0]$.
So $I_\gamma$ must contain $v_{p,0}$, but $\Lambda e_\gamma / \Lambda
v_{p,0}$ already has only four composition factors, so it must be the
Morita correspondent of $Z(\gamma)$.  The other correspondences follow similarly.
\end{proof}

The following table records a list of elements whose images form a basis for the modules $Z_\alpha$:

\begin{center}
    \begin{tabular}{ll}
        $Z_\gamma$: & $e_\gamma, v_{-p,p},
v_{0,-p}, w_{-p,p}v_{0,-p} = - w_{0,-p}v_{-p,p}$\\
        $Z_\mu$: & $ e_\mu, w_{-p,0}, w_{p,-p}, 
v_{p,-p}w_{-p,0} = - v_{-p,0}w_{p,-p}$\\
        $Z_\lambda$: &$e_\lambda, w_{0,-p}, v_{-p,0},
v_{0,-p}v_{-p,0}=-w_{-p,0}w_{0,-p}$.
    \end{tabular}
\end{center}

Put $I = \Lambda v_{0,p}+\Lambda w_{p,0} $ and $M = \Lambda e_\lambda /
I$, regarded as a $K$-graded module with $e_\lambda +I$ in degree zero.

\begin{lem}
  There is an exact sequence of $K$-graded modules
  \begin{equation}\label{M_seq} 0 \to Z_\gamma[1] \oplus Z_\mu [1]
    \stackrel i \to M \stackrel \pi \to Z_\lambda \to 0\end{equation}
  where $i(e_\gamma + \Lambda v_{p,0})=v_{p,-p}+I$, $i(e_\mu+\Lambda
  w_{0,p})= w_{-p,p} + I$  and $\pi$ is the quotient
  map. \end{lem}

\begin{proof}
  $i$ is well-defined because
  $w_{0,p}w_{-p,p}=-v_{-p,p}v_{0,p} \in I$ and $v_{p,0} v_{p,-p} =
  -w_{p,-p}w_{p,0} \in I$, and $\im i$ is the kernel of the quotient
  map $M \to Z_\lambda$.  We only need to show $i$ is injective,
  and it is enough to prove $i$ is injective on the socle $S_\mu
  \oplus S_\gamma$ of $Z_\gamma \oplus Z_\mu$.  Since the two simple
  summands are non-isomorphic (or by considering the weight degree), it is
  enough simply to show that $i$ is nonzero on $\soc Z_\gamma$ and on
  $\soc Z_\mu$.

  $i(\soc Z_\mu)$ is spanned by the image of $v_{p,-p}w_{-p,0}w_{-p,p}$
  in $M$, which is homogeneous of $K$-degree 3 and weight degree $[-p,0]$.
  We want to show this element does not lie in $I=\Lambda
  v_{0,p}+\Lambda w_{p,0}$.  We claim $\Lambda v_{0,p}$ has no nonzero
  element in this $K$- and
  weight-degree.  To get such an element, we would have to multiply $v_{0,p}$ by two of the
  generators $v_{i,j},w_{i,j}$, and the first has to be $v_{0,-p}$ as
  all other products are either zero or too far from $[-p,0]$ (it helps
  to look at Figure~\ref{Plambda} to see this).  The second must then be
  $v_{-p,0}$ if we are to end up at weight degree $[-p,0]$, but
  $v_{-p,0}v_{0,-p}v_{0,p} =0$ because $v_{-p,0}v_{0,-p}=0$.  Similarly,
  $\Lambda w_{p,0}$ has no nonzero $K$-degree 3 weight degree $[-p,0]$
  element: the only possible product is $v_{-p,0} w_{-p,0} w_{p,0}$ as
  before, but $v_{-p,0}w_{-p,0}=0$.

  A similar argument shows that $i$ is injective on $\soc Z_\gamma$,
  completing the proof.
\end{proof}

Let $\Omega (M) = (\Lambda e_\mu \oplus \Lambda
e_\gamma)/\Lambda(w_{0,p}+v_{p,0})$.

\begin{lem}
  There are exact sequences of $K$-graded modules
  \begin{gather} \label{OM_seq}0 \to Z_\lambda [1]\stackrel \psi \to \Omega(M)
    \stackrel \phi \to Z_\gamma \oplus
    Z_\mu \to 0 \\
  \label{M_res} 0 \to M[1] \stackrel i \to \Lambda e_\mu[1] \oplus \Lambda
  e_\gamma[1]
    \stackrel j \to \Lambda
    e_\lambda \to M \to 0
  \end{gather}
\end{lem}

\begin{proof}
	$\phi( (a,b)+\Lambda(w_{0,p}+v_{p,0})) = (a + \Lambda w_{0,p},
  b+\Lambda v_{p,0}) $ is well-defined and surjective, with kernel
  \[ \ker \phi = \frac{\Lambda v_{p,0} + \Lambda
  w_{0,p}}{\Lambda(v_{p,0}+w_{0,p})} \]
  generated by the image of $v_{p,0}$.  Therefore $Z_\lambda$ surjects
  onto $\ker \phi$ via the map $\psi$ induced by $e_\lambda \mapsto v_{p,0} +
  \Lambda(v_{p,0}+w_{0,p})$.  We only need $\psi$ to be injective.  If we
  show that $\dim \Omega(M) \geq 12$ then $\ker \phi$ has dimension at
  least four, but since it is a quotient of $Z_\lambda$ it will be
  exactly four and hence $\psi$ will have been shown to be injective.

Letting $\Lambda e_\lambda \to M$ be the quotient map and defining $j$
by $j(e_\mu) =w_{p,0}$, $j(e_\gamma) =v_{0,p}$, we get an exact sequence
\[ \Lambda e_\gamma[1] \oplus \Lambda e_\mu[1] \stackrel j \to \Lambda
e_\lambda \to M \to 0.\]
$v_{p,0}+w_{0,p}\in \ker j$ so $j$
factors through $\Omega(M)[1]$, showing that $\dim \Omega(M) \geq \dim \im
j = \dim \Lambda e_\lambda - \dim M = 12$.  This completes the proof of
the exactness of (\ref{OM_seq}), so $\dim \Omega(M)=12$.

$\dim \im j = 12$, so $\dim \ker j = 12$, and so
$\Lambda(v_{p,0}+w_{0,p})$, which has dimension 12 as $\Omega(M)$ does,
must be all of  $\ker j$.  Therefore $M[1]$ surjects onto $\ker j$ by
$i(e_\lambda+ I) =  v_{p,0}+w_{0,p}$ (this is well-defined).  Since
$\dim M =12$, $i$ must be injective.  This gives exactness of
(\ref{M_res}).
\end{proof}

\begin{cor} $M$ and $\Omega(M)$ are Koszul. \end{cor}
\begin{proof}
  Splicing shifted copies
  of (\ref{M_res}) produces a linear projective
 resolution of $M$. Koszulity of $\Omega(M)$ follows by truncating and
 shifting the
 same resolution.
\end{proof}

\begin{thm}
  $Z_\lambda, Z_\gamma$ and $Z_\mu$ are Koszul.
\end{thm}

\begin{proof}
  Let $S = S_\mu\oplus S_\lambda\oplus S_\gamma$. We will show that for every $n$, the graded vector spaces
  $\Ext^{n,m}_\Lambda(Z_\lambda, S)$ and
  \[\Ext^{n,m}_\Lambda(Z_\gamma \oplus Z_\mu, S) \cong
    \Ext^{n,m}_\Lambda(Z_\gamma , S)\oplus \Ext^{n,m}_\Lambda(Z_\mu,S)
  \]
  are zero unless $m=n$.  This is certainly true for $n=0$.

   Consider the long
  exact sequence obtained by applying $\Hom_\Lambda(-, S)$ to
  (\ref{M_seq}). Let $\omega$
   be the connecting homomorphism: it corresponds to Yoneda multiplication by
   the short exact sequence (\ref{M_seq}) and is therefore homogeneous
   of $K$-degree one. For each $n$ we get a
  short exact sequence
  \[ 0 \to \frac{\Ext^n_\Lambda(Z_\gamma \oplus Z_\mu, S)}{\ker \omega} \to
    \Ext^{n+1}_\Lambda(Z_\lambda,S) \stackrel{\pi^*}\to \im \pi^* \to
  0\]
in which the first map, induced by $\omega$, increases $K$-degree by one,
and the second preserves $K$-degree.  $M$ is Koszul, so $\im \pi^*
\subseteq \Ext^{n+1}_\Lambda(M, S)$ is zero outside $K$-degree $n+1$. It
follows that if $\Ext^n_\Lambda(Z_\gamma \oplus Z_\mu, S)$ is zero
outside $K$-degree $n$ then $\Ext^{n+1}_\Lambda(Z_\lambda, S)$ is zero
outside $K$-degree $n+1$.

Similar short exact sequences arising from the long exact sequence
obtained by applying $\Hom_\Lambda(-,S)$ to (\ref{OM_seq}) show
that if $\Ext^{n}_\Lambda(Z_\lambda, S)$ is zero outside $K$-degree
$n$ then $\Ext^{n+1}_\Lambda(Z_\gamma \oplus Z_\mu, S)$ is zero outside
$K$-degree $n+1$.
Since the result holds for $n=0$, this completes the proof.
\end{proof}

\section*{Acknowledgements} % (fold)
\label{sec:Acknowledgements}
I would like to thank the anonymous referee whose comments greatly
improved the presentation of the paper.

\bibliography{irregular_blocks}
\bibliographystyle{amsalpha}
\end{document}